\documentclass[11pt]{amsart}
\usepackage[T1]{fontenc}

\usepackage{lmodern, amsfonts,amsmath,amstext,amsbsy,amssymb,
amsopn,amsthm,upref,eucal,mathptmx}

\usepackage{mathrsfs}

\RequirePackage{xcolor} 
\definecolor{halfgray}
{gray}{0.55}
\definecolor{webgreen}
{rgb}{0,0.4,0}
\definecolor{webbrown}
{rgb}{.8,0.1,0.1}
\definecolor{red}
{rgb}{1,0,0}
\usepackage{microtype}

\newcommand \R {{ \mathbb R}}
\def\C{{\mathbb C}}
\newcommand \Z {{ \mathbb Z}}
\newcommand \N {{ \mathbb N}}
\newcommand \T {{ \mathbb T}}

\newtheorem{theorem}{Theorem}[section]
\newtheorem {lemma} [theorem]{Lemma}

\newtheorem{corollary}[theorem]{Corollary}
\newtheorem{remark}[theorem]{Remark}

\newtheorem{definition}[theorem]{Definition}

\newtheorem{addendum}[theorem]{Addendum}

\title[On equidistribution of unstable curves for pseudo-Anosov diffeomorphisms  ]%
{ On the equidistribution of unstable curves for pseudo-Anosov diffeomorphisms  of compact surfaces}

  \author{Giovanni Forni}

\address{Department  of Mathematics\\
  University of Maryland \\
  College Park, MD USA}
  
\email
    {gforni@math.umd.edu}

\date{\today}

\begin{document}

\def\echo#1{\relax}
    
\begin{abstract}
\begin{sloppypar}
We prove that the asymptotics of ergodic integrals along an invariant foliation of a toral Anosov 
diffeomorphism, or of a pseudo-Anosov diffeomorphism on a compact orientable surface of higher genus,
are determined (up to a logarithmic error)  by the action of the diffeomorphism on the cohomology of the surface. As a consequence of our argument and
of the results of Giulietti and Liverani \cite{GL} on horospherical averages, toral Anosov diffeomorphisms
have no  Ruelle resonances in the open interval $(1, e^{h_{top}})$.
\end{sloppypar}
\end{abstract}
\maketitle

\section{Introduction}

In this note we prove that the asymptotics of the equidistribution of unstable (or stable) curves for 
any $C^r$ ($r>1$) pseudo-Anosov diffeomorphism of a compact surface is entirely determined by
the action of the diffeomorphism on the first cohomology (or homology) group  up to a logarithmic error.
This work was motivated by the question on whether non-trivial resonances, in the interval
$(1, e^{h_{top}})$,  do appear in the asymptotics of ergodic integrals of Giulietti and Liverani \cite{GL} 
for Anosov diffeomoprhisms of the torus and in the spectrum of the relevant transfer operator.  By comparing 
our asymptotics with that of Giulietti and Liverani \cite{GL}, we conclude that no such non-trivial resonance exists. A direct, self-contained proof that there are no non-trivial resonances for the transfer operator has
been given simultaneously and independently by V.~Baladi \cite{Ba}. Her proof was an additional motivation to write up  the argument presented below. 

The argument is inspired by the author's proof \cite{F02} of deviation of ergodic averages for generic (almost all) translation flows on higher genus surfaces. Here we only deal with the special case of unstable foliations of diffeomorphisms, but we do not assume that the diffeomorphism is volume preserving. For this reason we 
work in H\"older spaces instead of Sobolev $L^2$ spaces (with respect to the invariant volume).  A complete
description of Ruelle resonances, as well as a complete asymptotics of ergodic averages and results on cohomological equations for linear pseudo-Anosov maps, has been given recently in the paper by Faure, 
Gou\"ezel and Lanneau \cite{FGL}. 
Our argument gives a simplified proof of the part of their result concerning the Ruelle
resonances in the interval $(1, e^{h_{top}})$ and the deviation of ergodic averages up to a logarithmic error. Our
argument also extends to the ``non-linear'' case, which to the best of our knowledge has not been studied so far.

\medskip

Let $A:M\to M$ be an orientation-preserving  pseudo-Anosov diffeomorphism of class $C^r$ for
any $r >1$ of a compact surface $M$, finite set of fixed points (singularities) at $\Sigma\subset M$, not necessarily linear (with respect to a translation structure
on $M$), but topologically conjugated to a linear pseudo-Anosov.

Let $E^+ \subset H^1(M, \C)$ and $E^-  \subset H^1(M, \C)$ denote respectively  the unstable and the stable spaces
of the finite dimensional linear map $A^\#: H^1(M, \C) \to  H^1(M, \C)$ induced by $A$ on cohomology. Let $\{\mu_1, \dots, \mu_{2s}\}\subset \C$  denote its spectrum with
$$
\mu_1=\lambda > \vert \mu_2 \vert \geq  \dots \geq \vert \mu_k\vert > \vert \mu_{k+1}\vert = \dots  \vert \mu_{s}\vert =1\,.
$$
Let $J_1 (=1), \dots, J_k \in \N\setminus \{0\}$ denote the geometric multiplicities of the expanding eigenvalues 
$\mu_1=\lambda >\mu_2 \geq \dots \geq \mu_k$ or, by symmetry, of the contracting eigenvalues $\mu_{2g}= \lambda^{-1} < \mu_{2g-1} \leq \dots \leq \mu_{2g-k+1}$ of the linear map $A^\#: H^1(M, \C) \to H^1(M,\C)$ and let 
$$\{\mathcal C^\pm_{i,j} \vert  i\in \{1,\dots, k\}, \,  j \in \{1, \dots, J_i\} \}$$  denote Jordan bases of the spaces $E^\pm\subset H^1(M,\C)$ respectively.  

Let $J^0_A$ denote the maximal geometric multiplicity of eigenvalues of $A$
on the unit circle, that is, of eigenvalues $\mu_{k+1}, \dots, \mu_{2g-k}$ of $A\vert E^0$.

Let $\mathcal L^\pm$  denote the conditional measures of the
Margulis measure (measure of maximal entropy) along the leaves of the unstable and, respectively,
stable foliations $\mathcal F^+$ and $\mathcal F^-$. For any $(x, \mathcal L) \in M\times \R^+$, let $\gamma_{\mathcal L}(x) \subset M\setminus \Sigma$ denote an unstable 
curve with initial point $x \in M$ and unstable Margulis ``length'' $\mathcal L^ +(  \gamma_{\mathcal L}(x)   ) 
= \mathcal L$ (whenever is exists). Since $A$ is topologically conjugated by assumption to a linear pseudo-Anosov diffeomorpshim, there exists
a set of full measure (with respect to the Margulis measure) of $x\in M$ such that $\gamma_{\mathcal L}(x) $ is
well-defined for all $\mathcal L>0$.
 
\subsection{Deviation of ergodic averages: invariant foliations}

Let  ${\mathcal Z}^{-1}(M)$ denote the subspace of closed currents of the space of currents of dimension $1$ (and degree $1$) dual to the space $\Omega^1(M)$ of $1$-forms of class $C^1$.  

\begin{theorem} 
\label{thm:Deviations} There exist injective maps ${\mathcal B}^\pm : E^\pm \to {\mathcal Z}^{-1}(M) \subset \Omega^1(M)^\ast$ into the subspace  ${\mathcal Z}^{-1}(M)$ of closed currents (of degree and dimension~$1$)  dual to the space $\Omega^1(M)$  of differential $1$-forms of class $C^1$ on $M$ such that
$$
A_* \circ {\mathcal B}^\pm =  {\mathcal B}^\pm  \circ  A^\#   \quad \text{ on}\quad  E^\pm \subset H^1(M, \C)\,.
$$
For $i\in \{1,\dots, k\}$ and $j \in \{1, \dots, J_i\}$ let us adopt the notation
$$
B^\pm_{i, j} :=  {\mathcal B}^\pm ( \mathcal C^\pm_{i, j} ) \in \Omega^{1}(M)^\ast \,.
$$
 The currents $B^+_1:= {\mathcal B}^+(\mathcal C^+_{1,1})$ and $B^-_1:= {\mathcal B}^-(\mathcal C^-_{1,1})$ and  can be explicitly written as follows: let $\mathcal M^+$ and $\mathcal M^-$ denote the unstable and the stable Margulis measures. For any  $1$-form $\alpha$ of class $C^1$  on $M$ we have
\begin{equation}
\label{eq:Margulis_currents}
B^+_1 (\alpha ) = \int_M  \mathcal M^- \otimes \alpha  \quad \text{ and } \quad B^-_1 (\alpha ) = \int_M \alpha \otimes \mathcal M^ + \,.
\end{equation}
There exists a constant $C>0$ such that the following hods. For any unstable curve $\gamma_{\mathcal L}(x)$  of initial point $x\in M$ and unstable length $\mathcal L>1$, there exists a set of uniformly bounded coefficients 
$$\{c_{i, j} (x, \mathcal L) \vert i\in \{2, \dots k\}, \, j\in \{1, \dots, J_i\} \}$$
such that, for any differential $1$-form $\eta \in C^1(M)$, we have 
\begin{equation}
\label{eq:deviations}
\begin{aligned}
\vert  \int_{\gamma_{\mathcal L}(x)} \eta - \mathcal L B^+_1(\eta) -  &\sum_{i=2}^k \sum_{j=1}^{J_i} 
c_{i,j} (x, \mathcal L) B_{i, j} ^+(\eta)  (\log \mathcal L)^{j-1} \mathcal L^{ \frac{ \log \vert \mu_i\vert}{h_{top}(A)}}  \vert  \\  &\leq C \Vert \eta \Vert_{C^1(M)}  [\log (1+ \mathcal L) ]^{\max(J^0_A,1) +1}\,.
\end{aligned}
\end{equation}
In addition, there exists $c>0$ such that for every $x\in M$ and for every $i\in\{2, \dots, k\}$ and $j\in \{1, \dots, J_i\}$, there exists a sequence $\mathcal L_n:=\mathcal L^{(i,j)}_n(x)$ such that 
\begin{equation}
\label{eq:lower_bounds}
\inf_{n\in \N} \vert  c_{i,j} (x, \mathcal L_n)  \vert \geq c  \,.
\end{equation}
\end{theorem} 
From the above asymptotic result of formula~\eqref{eq:deviations}, together with the lower bound of formula
\eqref{eq:lower_bounds}, we can derive {\it a posteriori} additional invariant properties of the closed currents
$B^\pm_1$, $B^\pm_{i,j}$ which are not apparent from their construction. In order to state these invariance properties we recall the following:
\begin{definition}  \label{def:basic}  A current $B$ is called \emph{basic} for a foliation $\mathcal F$ on a smooth
manifold  if, for all vector fields $Y$ tangent to $\mathcal F$,
$$
\mathcal L_Y B =  \imath_Y  B=0 \,.
$$
 If the current has dimension $1$, by the identity 
$\mathcal L_Y B =  \imath_Y  dB + d\imath_Y  B$ it follows that $B$  is basic if and 
only 
$$
d B =  \imath_Y  B=0\,.
$$
\end{definition}

\begin{addendum} 
\label{addendum:basic} The currents $B^\pm_1$ and $B^\pm_{i,j}$ for $i\in \{2, \dots, k\}$, $j\in \{1, \dots, J_i\}$,
of Theorem~\ref{thm:Deviations} are basic for the unstable, respectively, stable, foliations $\mathcal F^\pm$
on $M\setminus \Sigma$.
\end{addendum}
 
The asymptotic expansion of Theorem~\ref{thm:Deviations} can be refined by introducing finitely additive functionals on rectifiable arcs, following the work of A.~Bufetov~\cite{Bu14} on translation flows (see also \cite{BuFo14} on horocycle flows and \cite{FoKa} on nilflows). 

\begin{theorem}
\label{thm:Functionals} Let $\Gamma_r$ the set of all rectifiable paths (considered as a subset of the space of 
currents). There exists a  map $\hat \beta^+: \Gamma_r \to   \mathcal B^+(E^+)\subset \mathcal Z^{-1}(M)$ into the space of closed currents (with image in the space of basic currents for the unstable foliation) such that the following holds. The map $\hat \beta^+$ has the following properties:
\begin{enumerate}
\item (Additive property)   For any decomposition $\gamma= \gamma_1 +\gamma_2$ into subarcs,
\begin{equation*}
\hat \beta^+(\gamma)= \hat \beta^+(\gamma_1) +
\hat\beta^{+}(\gamma_2) \,;
\end{equation*}
\item (Scaling)  For any $\gamma \in \Gamma_r$, we have
\begin{equation*}
\hat\beta^{+}(A\gamma ) =  A_\ast \hat\beta^{+}(\gamma)  \,,
\end{equation*}
\item (Stable holonomy invariance)  For all pair of arcs $\gamma_1$, $\gamma_2\in \Gamma_r$ 
equivalent under the stable holonomy, we have
\begin{equation*}
\hat\beta^{+} (\gamma_1) =  \hat \beta^{+}(\gamma_2)\,.
\end{equation*}
\end{enumerate}
In addition, the functional $\hat \beta^+$ has the following asymptotic property: there exists a constant $C>0$
such that, for every arc $\gamma \in \Gamma_r$ we have
$$
\Vert  \gamma -  \hat \beta^+ (\gamma) \Vert_{-1} \leq C (1+\mathcal L^- ( \gamma))  [\log (1+\mathcal L^+ ( \gamma))]^{\max(J^0_A,1)+1}\,.
$$
\end{theorem}

\subsection{Deviation of ergodic averages:unstable (horocyclic) flows}

Let us then assume that the tangent map $DA$ of the pseudo-Anosov diffeomorphism preserves the orientation
of the unstable foliation of a pseudo-Anosov map $A$ of class $C^r$  with $r>1+\alpha$ and, as in~\cite{GL}, let $X$ denote a vector field of class $C^{1+\alpha}$, tangent to the unstable foliation and normalized to have constant norm with respect to a fixed Riemannian metric on $M$.  By definition, for all $n \in \N$ there  exists
a function $\nu_n: M\setminus \Sigma \to \R$ such that  (see \cite{GL}, formula (1.3)) 
$$
D_x A^n (X_x)=   \nu_n (x) X_{A^n(x)}  \,, \quad \text{ for all } x \in M\setminus \Sigma\,.
$$
From Theorem~\ref{thm:Deviations} we derive a result on the asymptotics of ergodic integrals for the flow
$h^X_\R$ generated by the unstable vector field $X$ on $M\setminus \Sigma$. 

For linear pseudo-Anosov maps on higher genus surfaces, a sharper asymptotics of ergodic integrals of the stable and unstable translation flows was obtained by F.~Faure, S.~Gou\"ezel and E.~Lanneau  \cite{FGL}, who also proved complete
results on the existence and regularity of solutions of the cohomological equation. 

For the case of toral Anosov diffeomorphisms the asymptotics of ergodic integrals of stable and unstable vector fields was studied in the pioneering  work of P.~Giulietti and C.~Liverani \cite{GL}. V. Baladi~\cite{Ba} has given a proof, independent of ours, that in the for toral case there are no ``deviation resonances''  in the Giulietti-Liverani asymptotics (see Remark \ref{rem:Deviations} below). Her argument also proves that for the stable or unstable vector fields of sufficiently regular Anosov  diffeomorphisms every zero average function is a continuous coboundary, a result which is beyond the reach of our cohomological approach.

Let $\hat X$ be a $1$-form of class $C^{1+\alpha}$ dual to the vector field $X$, in the sense that 
$$
\hat X (X) \equiv 1  \quad \text{ on } \,\,M\setminus \Sigma\,.
$$
The unique invariant probability measure $\mu_X$ of the flow $h^X_\R$ is given by the condition
$$
\mu_X \in \R^+ ( B^+_1 \wedge \hat X)   \quad  \left(  \text{defined as } \int_M f d\mu_X   = \frac{B^+_1 (f \hat X)}
{B^+_1 (\hat X)} \,, \text{ for all } f \in C^1(M) \right)\,.
$$
Let $\{D^X_{i,j} \vert i\in \{2, \dots, k\}, j \in \{1, \dots, J_i\} \}$ the finite set of distributions 
$$
D^X_{i,j} := B^+_{i,j}  \wedge \hat X \,, \qquad \text{ for all }   i\in \{2, \dots, k\}, j \in \{1, \dots, J_i\}\,.
$$
It follows from Addendum~\ref{addendum:basic} (see Section \ref{sec:RAsymptotics})  that, for all $i\in \{2, \dots,k\}$ and $j\in \{1, \dots, J_i\}$, the distributions  $D^X_{i,j}$ are $X$-invariant, in the sense
that 
$$
X D^X_{i,j} =0  \quad \text{ in }  \,\, \mathcal D' (M\setminus \Sigma).
$$
The following asymptotic expansion of ergodic integrals holds.
\begin{corollary}
\label{cor:Deviations}
 There exist a constant $C_X>0$ and, for all $(x,T) \in (M\setminus\Sigma) \times \R^+$, a finite set of uniformly bounded coefficients $$\{ c^X_{i,j} (x,T) \vert  i\in \{2, \dots, k\}, j \in \{1, \dots, J_i\} \}\,, $$   such that,  for all $f\in C^1(M)$ and for all $(x,T) \in 
 (M\setminus \Sigma)\times \R^+$ with  $h_{[0,T]} (x) \subset M\setminus \Sigma$, we have
 \begin{equation}
\label{eq:deviations_erg}
\begin{aligned}
\vert  \int_0^T  f \circ h^X_t(x) dt  - T  \int_M f d \mu_X  -  &\sum_{i=2}^k \sum_{j=1}^{J_i} 
c^X_{i,j} (x, T) D_{i, j} ^X(f)  (\log T)^{j-1} T^{ \frac{ \log \vert \mu_i\vert}{h_{top}(A)}}  \vert  \\  &\leq C_X \Vert f\Vert_{C^1(M)}  [\log (1+ T) ]^{\max(J^0_A,1) +1}\,.
\end{aligned}
\end{equation}
 In addition, there exists $c>0$ such that for every $x\in M$ and for every $i\in\{2, \dots, k\}$ and $j\in \{1, \dots, J_i\}$, there exists a sequence $T_n:=T^{(i,j)}_n(x)$ such that 
\begin{equation}
\label{eq:lower_bounds_erg}
\inf_{n\in \N} \vert  c_{i,j} (x, T_n)  \vert \geq c  \,.
\end{equation}

\end{corollary}

\begin{remark}  \label{rem:Deviations} (Comparison with a result of V.~Baladi~\cite{Ba}, Corollary 2.3) 
A similar, but more refined, asymptotics of ergodic averages is proved in \cite{Ba} for the case of $C^r$ Anosov diffeomorphisms of the $2$-torus,
a case for which the spectrum of $A^\#: H^1(\T^2)\to H^1(\T^2)$ has a unique expanding eigenvalue $\mu_1=\lambda>1$ (and no neutral eigenvalues).

 In fact, she proves (see \cite{Ba}, Corollary 2.3) that there exists $r_0, r_1>1$ such that for any $r\geq \max\{r_0, r_1\}$  there exist  constants $C>0$  and $\theta_{min} <0$ such that for  all $f \in C^{r-1}(\T^2)$, and for all  $T>0$,
$$
\left \vert \int_0^T f \circ h^X_t (x) dt - T \int_{\T^2} f d\mu_X  \right\vert \leq C   ( T^{\theta_{min}} 
\vert f  \vert_{C^{r-1}}  + \sup \vert f \vert) \,.
$$
In particular, $\mu_X(f) =0$ if and only if $f$ is a continuous coboundary.

In comparison with the asymptotics of Corollary~\ref{cor:Deviations} above, Baladi's asymptotics has bounded,
not logarithmic, error terms. However, such terms cannot be neglected for general functions of class $C^1$. A refined asymptotics with
bounded error terms may hold for general functions of class $C^{1+\alpha}$, for some $\alpha>0$, and for such functions one can derive results on existence of solutions of the cohomological equation (see~\cite{F07} 	and
\cite{MY16} also in the higher genus case). Such refined results are beyond the purely cohomological approach presented here.

\end{remark} 

\subsection{Ruelle-Pollicott asymptotics}

From the above equidistribution results we derive a Ruelle-Pollicott asymptotics for pseudo-Anosov maps.
For linear pseudo-Anosov maps on higher genus surfaces, a complete Ruelle-Pollicott asymptotics has been
obtained by F.~Faure, S.~Gou\"ezel and E.~Lanneau~\cite{FGL}. For toral Anosov diffeomorphisms the Ruelle-Pollicott asympotics follows from the work of P.~Giulietti and C.~Liverani \cite{GL}. V.~Baladi~\cite{Ba} has given an independent proof that in the toral case there are no ``deviation resonances''  in the Giulietti-Liverani 
asymptotics (see Remark \ref{rem:RAsymptotics} below).

\begin{corollary} 
\label{cor:RAsymptotics} Let us assume that the diffeomorphism $A$ is of class  $C^r$ with $r>2$.  Then  $A$ has a Ruelle-Pollicott asymptotics  in the sense that, for any $f, g\in C^1(M)$, the correlations $\langle f \circ A^n , g \rangle$ with respect to the Margulis measure $\mathcal M$ have an expansion
 \begin{equation}
\label{eq:RuelleA}
\begin{aligned}
\langle f \circ A^n , g \rangle &= (\int_M f d\mathcal M) (\int_M g d\mathcal M)  \\ &+ \sum_{i=2}^k \sum_{j=1}^{J_i} \mathcal C_{i,j}(f,g,n) n^{j-1} \left(\frac{\mu_i}{e^{h_{top}(A)}}\right)^n + \mathcal R(f,g,n) \frac{n^{\max(J^0_A,1)+1}}{e^{n h_{top}(A)}} \,,
\end{aligned}
\end{equation}
with $\mathcal C_{i, j} (f,g,n)$ and $\mathcal R(f,g,n)$ uniformly bounded: there exists a constant $C>0$ such that,
for all $n\in \N$ we have 
$$
\sum_{i=2}^k \sum_{j=1}^{J_i} \vert \mathcal C_{i, j} (f,g,n) \vert + \vert \mathcal R(f,g,n) \vert \leq C \vert f \vert_{C^1(M)} \vert g \vert_{C^1(M)}\,.
$$
\end{corollary} 

\begin{remark} (Comparison with a result of V.~Baladi~\cite{Ba}, Corollary 2.5)  
\label{rem:RAsymptotics} A similar, but more refined, Ruelle-Pollicott asymptotics is proved in \cite{Ba} for the case of $C^r$ Anosov diffeomorphisms of the $2$-torus,
a case for which the spectrum of $A^\#: H^1(\T^2)\to H^1(\T^2)$ has a unique expanding eigenvalues $\mu_1=\lambda>1$ (and no neutral eigenvalues).

 In fact, she proves (see \cite{Ba}, Corollary 2.5) that for any $r>1$  there exist  constants $C>0$  and $\rho \in(0,1)$ such that for  all $f ,g \in C^{r-1}(\T^2)$, and for all  $n\in \N$,
$$
\left \vert \langle f \circ A^n , g \rangle - (\int_{\T^2} f d\mathcal M) (\int_{\T^2} g d\mathcal M) \right\vert \leq C \rho^n \vert f  \vert_{C^{r-1}} \vert g  \vert_{C^{r-1}} \,.
$$
The error term in the above estimate is refined as follows. Let $\lambda^+ > 1 > \lambda^-$ denote the Lyapunov
exponents. There exists $\tilde \rho_A < e^{h_{top}(A)} \min (\lambda^+, (\lambda^-)^{-1})^{-\frac{r-1}{2}}$  such that for $\tilde \rho_A<1 $ the estimate holds for some $\rho< e^{-h_{top}(A)}$. In particular, there exists $r_1>1$ such that whenever $A$ is of class $C^r$ for $r>r_1$ then $\tilde \rho_A<1$ and the estimates hold for some $\rho < e^{-h_{top}(A)}$. For $\tilde \rho_A \geq 1$ the estimate holds for any $\rho > e^{-h_{top}(A)} \tilde \rho_A$.

\end{remark}
The paper is organized as follows. In Section~\ref{sec:closed_curr} we prove a representation lemma for stable and unstable cohomology classes in terms, respectively, of stable and unstable closed (basic) currents in the dual space of $1$-forms of class $C^1$ (Lemma \ref{lemma:rep}).
From the representation lemma and de Rham theorem we derive a result on the asymptotics of the action of
the pseudo-Anosov diffeomorphism on the space of closed currents dual to the space of $1$-forms of class $C^1$. In Section~\ref{sec:return_leaves} we apply the results of Section~\ref{sec:closed_curr} first to a special sequence of return leaves of the unstable foliation, then after proving a standard decomposition lemma (see Lemma~\ref{lemma:decomposition}), to all leaves. We then complete the proof of the main result, Theorem~\ref{thm:Deviations}. Finally, in Section \ref{sec:RAsymptotics} we derive the proof of Corollary~\ref{cor:Deviations}  on the deviation of ergodic averages and Corollary~\ref{cor:RAsymptotics} on the Ruelle-Pollicott asymptotics. 

 \section*{Acknowledgements} The author wishes to thank  V.~Baladi, P.~Giulietti, S.~Gou\"ezel and C.~Liverani for their interest in the cohomological perspective presented here and their encouragement to write up this note.
 He is also grateful to the participants of the Oberwolfach Seminar on {\it Anisotropic Spaces and their Applications to Hyperbolic and Parabolic Systems} (June 9-15, 2019)  for the stimulating environment and to MFO for making  the Seminar possible. Part of the paper took shape while the author was preparing his lectures at the Seminar.  This research was supported by the NSF grant DMS 1600687 and by a Research Chair of  the Fondation Sciences Math\'ematiques de Paris (FSMP). The author is grateful to the Institut de 
 Math\'ematiques de Jussieu (IMJ) for its hospitality during the preparation of the first draft of the paper.

\section{Growth of closed currents}
\label{sec:closed_curr}
\begin{lemma}
\label{lemma:rep}
There exist injective maps ${\mathcal B}^\pm : E^\pm \to {\mathcal Z}^{-1}(M) \subset \Omega^1(M)^\ast$ into the subspace  ${\mathcal Z}^{-1}(M)$ of closed currents (of degree and dimension~$1$)  dual to the space $\Omega^1(M)$ of  $1$-forms of class $C^1$ on $M$ such that
$$
A_* \circ {\mathcal B}^\pm =  {\mathcal B}^\pm  \circ  A^\#    \quad \text{ on}\quad  E^\pm \subset H^1(M, \C)\,.
$$
\end{lemma}  
\begin{proof}
 There exists a linear map $\alpha: H^1(M, \C) \to Z^1(M) \subset \Omega^1(M)$ of closed form of class
$C^1$ on $M$ such that, for any $C\in H^1(M, \C)$ we have that $[\alpha(C)] = C \in  H^1(M,\C)$, hence
there exists a (bounded) linear map $u: H^1(M, \C) \to C^{r+1}(M)$ such that 
$$
A^\ast \circ \alpha  = \alpha \circ A^\# +  d u   \quad \text{ on} \quad   H^1(M, \C) \,.
$$
We note that since the exterior derivative is elliptic, for every $C\in H^1(M, \C)$, the function $u(C) \in C^{r+1}$ under the hypothesis  that $A$ is a $C^r$ diffeomorphism.

By iterating the above identity, for any $C\in H^1(M, \C)$, we have
$$
(A^\ast)^n (\alpha(C)) = \alpha [{(A^\#)^n(C)}] +  d  [A^\ast (u ((A^\#)^{n-1}(C))]     + \dots   +  d [(A^\#)^n \circ u(C)]\,.
$$
It follows that 
$$
(A^\ast)^n \circ  \alpha \circ (A^\#)^{-n}   =  \alpha +   d  \circ A^\ast \circ u \circ ( A^\#)^{-1}      + \dots   +  
d  \circ (A^\ast)^n  \circ u  \circ (A^\#)^{-n}  \,.
$$
Let us give the argument for the unstable space $E^+ \subset H^1(M,\C)$, otherwise we replace $A$ with it inverse $A^{-1}$. 
We claim that, since the restriction $A^\#\vert E^+$ is (strictly) expanding, it follows  by completeness that
for every $C\in E^+ \subset H^1(M, \C)$, the following limit exists in $C^0(M)$ (and in $L^2(M)$ 
in the volume preserving case):
$$
\begin{aligned}
U (C) &:=   (A^\ast \circ u) (( A^\#)^{-1} C)     + \dots   +  
 ((A^\ast)^n  \circ u)  ((A^\#)^{-n}C)  \\ &= \sum_{k=1}^\infty   ((A^\ast)^k \circ u) ( (A^\#)^{-k}C )  \in C^0(M)\,.
\end{aligned}
$$
In fact, for every function $u \in C^0(M)$ we have that 
$$
\Vert  (A^\ast)^n (u)    \Vert_{C^{0}(M)} =  \Vert  u   \Vert_{C^{0}(M)}   
$$
and, in the volume preserving case, also
$$
\Vert  (A^\ast)^n (u)    \Vert_{L^{2}(M, \text{vol} )} =  \Vert  u   \Vert_{L^{2}(M, \text{vol} )}   \,.
$$
Thus, since $A^\#\vert E^+$ is (strictly) expanding,  there exists $\rho >1$ such that, for all $k\in \N$, 
$$
\Vert  ((A^\ast)^k \circ u) ( (A^\#)^{-k}C )  \Vert_{C^0(M)} \leq  \left( \max_{\vert C\vert=1}
 \Vert u(C) \Vert_{C^0(M)} \right)  \rho^{-k} \,,
$$
which implies the absolute convergence of the series, hence the claim.

It follows that the following limit exists in $(C^1(M))^\ast$:
$$
\mathcal B^+(C):= \lim_{n\to + \infty}  (A^\ast)^n \circ \alpha \circ (A^\#)^{-n} (C) =  \alpha (C) + d U(C) \,.
$$ 
By construction we clearly have $[\mathcal B^+(C)] = [\alpha (C)] = [C] \in H^1(M,\C)$ and 
$$
\begin{aligned}
A_\ast (\mathcal B^+(C) )& = \lim_{n\to + \infty}  (A^\ast)^{n+1} \left(\alpha ( (A^\#)^{-n}  (C)    \right) \\ & =  
 \lim_{n\to + \infty}  (A^\ast)^{n+1} \left( \alpha ( (A^\#)^{-(n+1)}  ( A^\# C)    \right) = 
 \mathcal B^+ ( A^\# C)\,.
\end{aligned}
$$
The map $\mathcal B^+: E^+ \to \mathcal Z^{-1}(M)$ is therefore defined, it is linear by its definition and it is injective since $[\mathcal B^+(C)] =[C] \in H^1(M, \C)$.
\end{proof}

\begin{remark} By definition of the Margulis measure, the currents in formula   are distributional eigenvectors for the diffeomorphism $A$ for the eigenvalues $\lambda^{\pm1}$ since
$$
\begin{aligned}
A_\ast B^+_1 (\alpha)&= B^+_1 ( A^\ast \alpha ) = \int_M \mathcal M^- \otimes A^\ast \alpha  = \int_M  A^{-1}_\ast \mathcal M^- \otimes  \alpha = \lambda B^+_1 (\alpha ) \,, \\
A_\ast B^-_1(\alpha)&=  B^-_1( A^\ast \alpha ) = \int_M A^\ast \alpha \otimes \mathcal M^+  = \int_M   \alpha \otimes  A^{-1}_\ast \mathcal M^+
= \lambda^{-1} B^-_1 (\alpha )  \,.
\end{aligned}
$$
Since the eigenvalues $\lambda^{\pm 1}$ for action $A^*$ of the diffeomorphism $A$ on $H^1(M, \C)$ are simple, it follows that  the Margulis currents $B^\pm_1$  in formula~\eqref{eq:Margulis_currents} are (up to multiplicative constants) the unique distributional eigenvectors of eigenvalues $\lambda^{\pm 1}$  for the linear map $A_\ast$ on the space of closed
$1$-currents. 
\end{remark} 

Let us recall that $\mathcal Z^{-1}(M)$ denotes the space of closed currents of dimension and degree $1$ on $M$ dual to the space of $1$-forms of class $C^1$.  Let $E^+$, $E^-$ and $E^0$ denote, respectively, the unstable, the stable and the central stable space of the linear map $A^\#: H^1(M, \C) \to H^1(M, \C)$ induced by $A:M\to M$ on the first cohomology of $M$.  There is a direct decomposition
\begin{equation}
\label{eq:cohom_split}
H^1(M, \C)= E^+ \oplus E^- \oplus E^0\,.
\end{equation}
By Lemma \ref{lemma:rep} there exist  maps $\mathcal B^\pm: E^\pm \to \mathcal Z^{-1}(M)$ such that 
$$
A^* \circ \mathcal B^\pm =  \mathcal B^\pm \circ A^\#    \quad \text{ on} \quad E^\pm\,.
$$
There is also a linear map $\mathcal B^0: E^0\to  {\mathcal Z}^1 (M)$,  with values in the space ${\mathcal Z}^1 (M)$ of closed smooth $1$-forms, and  a linear map $F: E^0 \to C^\infty(M)$ such that 
$$
A^* \circ \mathcal B^0 =    \mathcal B^0 \circ A^\#   +   d F  \quad \text{ on} \quad E^0\,.
$$
The restriction $A^\# \vert E^0$ is by definition a unipotent linear operator. Let $J^0_A$ the dimension of the
largest Jordan block of $A^\# \vert E^0$.

For every closed current $\gamma \in \mathcal Z^{-1}(M)$ let $[\gamma] \in H^1(M, \R)$ denote its cohomology
class and let $[\gamma]^+$, $[\gamma]^-$ and $[\gamma]^0$ denote, respectively, the projections of the cohomology class $[\gamma]$ on the subspaces $E^+$, $E^-$ and $E^0$ according to the decomposition 
in formula~\eqref{eq:cohom_split}, that is, for all $\gamma \in \mathcal Z^{-1}(M)$ we have 
$$
[\gamma] = [\gamma] ^+  + [\gamma]^- + [\gamma] ^0 \,, \quad \text{ with } [\gamma]^\pm \in E^\pm \text{ and }
[\gamma]^0 \in E^0\,.
$$

\begin{lemma}
\label{lemma:closed_currents}
There exists $C>0$ such that, for any closed current $\gamma \in \mathcal Z^{-1}(M)$ of dimension $1$ (and degree $1$)  and for any $n\in \N$, we have
$$
\Vert  A_\ast^n(\gamma) -  \mathcal B^+\left( (A^\#)^n [\gamma]^+ \right)   \Vert_{-1} \leq C \Vert [\gamma] \Vert\, n^{\max(J^0_A,1)}\,.
$$
\end{lemma}

\begin{proof}
For every $n\in \N$, there exist  $X^\pm(n) \in E^\pm$, $X^0(n) \in E^0$ and a current  $U_n$ of dimension
$2$ (and degree $0$) in the dual space of the space of  $2$-forms with $C^0$ coefficients,  such that we can write
$$
A^n_\ast  (\gamma)=\mathcal B^+ (X^+(n))   + \mathcal B^- (X^-(n))   + \mathcal B^0 (X^0(n) ) + dU_n\,.
$$

We therefore have the identities
$$
\begin{aligned}
A^{n+1}_\ast  (\gamma)&=\mathcal B^+ (X^+(n+1))   + \mathcal B^- (X^-(n+1))   + \mathcal B^0 (X^0(n+1) ) + dU_{n+1}  \\ &=   A_\ast \left( \mathcal B^+ (X^+(n))   + \mathcal B^- (X^-(n))   + \mathcal B^0 (X^0(n) ) + dU_{n}   \right)  \\ &=  \mathcal B^+ (A^\# X^+(n))   + \mathcal B^- (A^\# X^-(n))   + \mathcal B^0 (A^\# X^0(n) ) + 
 dF (X^0(n)) + dA^\ast U_{n}   
 \,. 
\end{aligned}
$$ 
By projecting the above identity on cohomology  we have
$$
\begin{aligned}
X^\pm (n+1)  &=  A^\# X^\pm(n),   \quad  X^0(n+1) =  A^\# X^0(n) \quad \text{ and }
\\ &U_{n+1}= F(X^0(n)) +   A^\ast(U_{n})\,,
\end{aligned}
$$
from which we derive that 
$$
\begin{aligned}
&X^\pm_j(n)  =  (A^\#)^n( x^\pm (0))= (A^\#)^n( [\gamma]^\pm)  ,  \\ & X^0 (n) =  (A^\#)^n(X^0 (0))= (A^\#)^n  
([\gamma]^0) \,,
\end{aligned}
$$
and that there exists a constant $C>0$ such that 
$$
 \Vert U_n\Vert_{C^0(M)^*}  \leq   C \Vert [\gamma]^0 \Vert\,  n^{\max(J^0_A,1)} \,.
$$
In fact
$$
\begin{aligned}
\Vert U_{n+1}\Vert_{C^0(M)^*} \leq  \Vert A^\ast(U_{n}) \Vert_{_{C^0(M)^*}} + C_F \Vert X^0 (n) \Vert  \\  \leq  \Vert U_{n} \Vert_{C^0(M)^*} + C'_F  (n+1)^{\max(J^0_A,1)-1} \Vert X^0(0) \Vert \,,
\end{aligned}
$$
hence
$$
\Vert U_{n}\Vert_{_{C^0(M)^*}} \leq C'_F \Vert [\gamma]^0\Vert   \sum_{m=0}^n   (m+1)^{\max(J^0_A,1)-1} \,,
$$
thus the argument is concluded.
\end{proof}

\section{Closest return leaves}
\label{sec:return_leaves}
In this section we derive the asymptotics for closed currents given by closest return leaves of  the unstable
foliation. We measure distances on $M$ along the  unstable and stable foliations by the conditional measures
$\mathcal L^\pm$  of the Margulis measure. 

For all $x\in M$ let $I^s(x)$ denote a stable curve centered at $x\in M$ of unit length and let $\bar \gamma(x)$ 
be an unstable return curve with endpoints $x, y \in I^s(x)$ with no other intersections with $I^s(x)$. 
Let $\gamma(x)= \bar \gamma(x) \cup J^s(x,y)$ denote the closed arc which is union of the unstable curve
$\bar \gamma(x)$ with the stable curve $J^s(x,y)\subset I^s(x)$ with endpoints $x,y \in I^s(x)$. 
By compactness there exist constants $C >c >0$ such that 
$$
c  \leq \mathcal L ^+(\bar \gamma(x)) \leq  C \,.
$$
For every $n\in \N$, let then $\bar \gamma_n (x)$ and $\gamma_n (x)$ be the curves
$$
\bar \gamma_n(x) = A^n \left( \bar \gamma( A^{-n}(x)) \right) \quad \text{ and } \quad 
\gamma_n(x) = A^n \left( \gamma( A^{-n}(x) ) \right) \,.
$$
Let $\lambda= e^{h_{top}(A)}$ denote the expansion rates of the diffeomorphism $A$ on $M$ with respect to the unstable Margulis measure. We have 
$$
c\lambda^n    \leq \mathcal L^ + (\bar \gamma_n(x)) \leq  C  \lambda^n \,.
$$
We have the following decomposition result:
\begin{lemma}
\label{lemma:decomposition}
Any unstable curve $\bar \gamma$ has a decomposition into consecutive closest returns 
closed curves,
\begin{equation}
\label{eq:chop} 
\bar \gamma = \sum_{\ell=1}^n \sum_{m=1}^{m_\ell} \bar \gamma_\ell( x_{\ell,m}) \,\, + \,\, 
\gamma_0 \,, 
\end{equation}
such that $c\lambda^n    \leq \mathcal L^+ (\bar \gamma) \leq C \lambda^{n+1}    $ and,  for all $1\leq \ell \leq n$, 
\begin{equation}
\label{eq:chopest}
m_\ell \leq C \lambda/c \,,  \quad  \mathcal L ^+ ( \bar \gamma_\ell( x_{\ell,m})) \in [c\lambda^\ell, C\lambda^\ell]  \quad \text{\rm and} \quad
 \mathcal L ^+ (\gamma_0) ) \leq  C\lambda/c\,.
\end{equation}
\end{lemma}
\begin{proof}
Let $x\in M$ denote the initial point of $\bar \gamma$. Let $n\in \N$ denote the unique integer such that 
$\bar \gamma_n(x) \subset \bar \gamma$. We have
$$
 c  \lambda^n \leq \mathcal L^+  (\bar \gamma_n(x))  \leq \mathcal L^+ (\bar \gamma) \leq \mathcal L^+  (\bar \gamma_{n+1}(x)) \leq C \lambda^{n+1}\,.
$$
Let then $x_{n,1}=x$.  Let $y$ denote the endpoint of $\bar \gamma_n(x)$.  Let then $n'$ denote the largest 
integer such that $\bar \gamma_{n'}(y) \subset \bar \gamma \setminus \bar \gamma (x)$. If $n'=n$, we set $y =x_{n,2}$, otherwise if $n' <n$ we set $y= x_{n',1}$. We keep iterating this procedure. It is clear that 
$$
c  m_\ell \lambda^\ell \leq   C \lambda^{\ell+1} 
$$
otherwise the union of segments $\bar \gamma_\ell( x_{\ell,1}) \cup \dots \bar \gamma_\ell( x_{\ell,m_\ell})$ would
be covered by a segment $\bar \gamma_{\ell+1}( x_{\ell+1,m_{\ell+1}})$.  In particular, we have that 
$\mathcal L ^+ (\gamma_0) \leq  C\lambda/c$.

\end{proof} 

\begin{theorem}  
\label{thm:asymptotic} There exists  a map $\mathcal C^+: \Gamma^+ \to E^+$  on the set $\Gamma^+
$ of unstable curves with bounded range in $E^+\subset H^1(M, \C)$, such that for any unstable curve $\bar \gamma$ of unstable length $\mathcal L^+(\bar \gamma)>1$ we have
\begin{equation}
\label{eq:asymptotic}
\Vert  \bar \gamma -  \mathcal B^+ (  (A^\#)^{[\frac{\log (\mathcal L^+(\bar\gamma))} {h_{top}(A)}]} ( \mathcal C^+(\bar\gamma) )   ) \Vert_{-1} \leq C [\log (1+\mathcal L^+ (\bar \gamma))]^{\max(J^0_A,1)+1}\,.
\end{equation}
\end{theorem}
\begin{proof}
For any $x, y\in M$, let $I_{x,y}$ denote a curve of bounded length joining $x$ to $y$. For any unstable
curve $\bar \gamma$ with endpoints $x$ and $y\in M$, let $\gamma = \bar \gamma \cup I_{x,y}$. 
There exists a current $u_\gamma$ of degree $0$ (and dimension $2$)  in the dual space of the space of
$2$-forms with $C^0$ coefficients
$$
\gamma = \mathcal B^+([\gamma]^+) +  \mathcal B^-([\gamma]^-)  + \mathcal B^0([\gamma]^0) 
+ du_\gamma\,.
$$
By the decomposition lemma we also have
$$
\bar \gamma = \sum_{\ell=1}^n \sum_{m=1}^{m_\ell} \bar \gamma_\ell( x_{\ell,m}) \,\, + \,\, 
\bar \gamma_0\,,
$$
with $ \bar \gamma_\ell( x_{\ell,m}) = A^\ell_\ast \left( \gamma ( A^{-\ell} (x_{\ell,m}) )       \right)$, hence 
there exists a constant $C>0$ such that 
$$
\Vert  \bar \gamma_\ell( x_{\ell,m}) -
 \mathcal B^+ \left( (A^\#)^\ell [ \gamma ( A^{-\ell} (x_{\ell,m})]^+  \right) \Vert_{-1} \leq 
 C \ell ^{\max(J^0_A,1)}  \,,
 $$
and by summation
$$
\Vert  \bar \gamma -
  \sum_{\ell=1}^n \sum_{m=1}^{m_\ell}  \mathcal B^+ \left( (A^\#)^\ell [ \gamma ( A^{-\ell} (x_{\ell,m})    ]^+  \right)     \Vert_{-1} \leq 
 C n^{\max(J^0_A,1)+1}\,.
$$
Let $\mathcal C^+(\bar\gamma) \in E^+$ denote the cohomology class
\begin{equation}
\label{eq:Cplus}
\mathcal C^+(\bar \gamma) := \sum_{\ell=1}^n\sum_{m=1}^{m_\ell}  (A^\#)^{-( [ \frac{\log (\mathcal L^+(\bar \gamma))} {h_{top}(A)}] -\ell)}  [ \gamma ( A^{-\ell} (x_{\ell,m}) )   ]^+  \,.
\end{equation}
Since by the decomposition lemma we have 
$$ n\leq  \frac{\log (\mathcal L^+(\bar \gamma)/c)} {\log \lambda} =  \frac{\log (\mathcal L^+(\bar\gamma)/c)} {h_{top}(A)}\,,  $$
and since $(A^\#)^{-1}$ is contracting on the unstable space $E^+$, by the decomposition lemma, it follows that
there exists a constant $C'>0$ such that
$$
\Vert  \mathcal C^+(\bar \gamma) \Vert \leq   C'  \,, \quad \text{ for all } \bar \gamma \in \Gamma^+\,.
$$
 The result follows.
\end{proof}

\begin{proof}[Proof of Theorem~\ref{thm:Deviations}] It follows from Theorem~\ref{thm:asymptotic}
by writing the asymptotics in formula \eqref{eq:asymptotic}  with respect to a Jordan basis of $A^\# \vert E^+$.
The lower bound on the coefficients along subsequences in formula~\eqref{eq:lower_bounds} holds by Lemma~\ref{lemma:closed_currents} along sequences of closest return leaves.  In fact, it is known that  for every projection given by  a Jordan basis of $A^\#$ in $H^1(M, \R)$ and  for all $x\in M\setminus \Sigma$ with infinite forward orbit, there exists a close return orbit $\gamma(x)$ with non-zero projection (see for instance \cite{Bu14}, Prop. 2.9, or \cite{DHL14}, \S~5.3). We recall below a proof of this statement for the
convenience of the reader.  

Let $I$ denote any stable arc. The return map of the unstable flow $h^X_\R$ to $I$ is an Interval Exchange Transformation  and that the flow $h^X_\R$ is a suspension of that Interval Exchange Transformation under a piece-wise constant roof function. 

It is well-known that the set of loops equal to the union of a first return orbit of $h^X_\R$ (to any given transverse stable arc $I$) with a transverse (stable) segment joining the endpoints
spans the relative homology group  $H_1(M \setminus \Sigma, \Z)$ (see for instance \cite{Yoc10}, \S 4.5). By Poincar\'e duality, it follows that, for every given Jordan projection (induced by a Jordan basis of the map $A^\#$ on $H^1(M , \R)$), there exists a subinterval  $J \subset I \subset M$ such that the cohomology class of the loop $\gamma_I(x)$, given by the union  of the return orbit of any point $x\in J$ with a transverse (stable) arc $I_x\subset I$, has a non-zero projection.  

Let now $x \in I \subset M\setminus \Sigma$ be an arbitrary point with  infinite forward orbit. Let  
$\bar {\gamma}'(x)$ denote the orbit arc until the first visit of the forward orbit $h^X_{\R^+}(x)$ to the subinterval
$J$ and let $\gamma'(x)$ denote the loop given by the union of $\bar {\gamma}'(x)$ with a transverse (stable) interval $I' \subset I$. Let $\bar {\gamma}''(x)$ denote  the smallest orbit arc  of the forward orbit $h^X_{\R^+}(x)$ which is strictly larger then $\bar {\gamma}'(x)$ and let $\gamma''(x)$ denote the the loop given by the union of $\bar {\gamma}''(x)$ with a transverse (stable) interval $I''\subset I$.  

{\it We claim that the given Jordan projection of either the cohomology class $[\gamma'(x)]\in H^1(M, \Z)$ of the loop $\gamma'(x)$ or cohomology class $[\gamma''(x)]\in H^1(M, \Z)$ of the loop $\gamma''(x)$ is non-zero.}

In fact, let $x'\in J \subset I$ denote the endpoint of the orbit arc $\bar {\gamma}'(x)$ and let $\gamma_I(x')$
denote the loop given by the union  of the return orbit of $x'\in J$ with a transverse (stable) arc $I_{x'}\subset I$.
On the one hand, by hypothesis we have that the given Jordan projection of the class $[\gamma_I(x')]$ is non-zero; on the other hand, by construction we have
$$
[\gamma''(x)] = [\gamma'(x)] +  [\gamma_I(x')]    \quad \text{ in } \,\,  H^1(M, \Z)\,,
$$
hence the claim follows. 

The lower bounds in the statement of Theorem~\ref{thm:Deviations} follow from the above claim on Jordan projections of orbit arcs and from Theorem~\ref{thm:asymptotic}.
\end{proof}

From the statement of Theorem~\ref{thm:Deviations} we derive:

\begin{proof}[Proof of Addendum~\ref{addendum:basic}]
Since the currents $B^\pm_1$ and $B^\pm_{i,j}$ are closed, it is enough to prove that, for every
vector field $Y^\pm$ tangent to the unstable/stable foliation $\mathcal F^\pm$, we have that 
$$
\imath_{Y^\pm} B^\pm_1 = \imath_{Y^\pm} B^\pm_{i,j} =0  \,, \quad \text{ for all }i\in \{2, \dots, k\}, j\in \{1, \dots, J_i\}\,.
$$
By Theorem~\ref{thm:Deviations} for any vector field $Y:=Y^+$ tangent to the unstable foliation, and for every
$2$-form $w$ of class $C^1$ we have, since $\int_{\gamma_{\mathcal L} (x)} \imath_{Y} w =0$,
\begin{equation}
\begin{aligned}
\vert \mathcal L B^+_1(\imath_{Y} w) -  &\sum_{i=2}^k \sum_{j=1}^{J_i} 
c_{i,j} (x, \mathcal L) B_{i, j} ^+(\imath_{Y} w)  (\log \mathcal L)^{j-1} \mathcal L^{ \frac{ \log \vert \mu_i\vert}{h_{top}(A)}}  \vert  \\  &\leq C \Vert \imath_{Y} w \Vert_{C^1(M)}  [\log (1+ \mathcal L) ]^{\max(J^0_A,1) +1}\,.
\end{aligned}
\end{equation}
The above inequality immediately implies that $(\imath_{Y} B^+_1)(w)=B^+_1(\imath_{Y} w) =0$, then by the lower bound on
coefficients $c_{i,j} (x, \mathcal L)$ given in Theorem~\ref{thm:Deviations}, it also follows by finite induction 
on $(i,j)$, with respect to the lexicographic order (such that $(i,j)<(i',j')$ iff $i<i'$ or $i=i'$ and $j>j'$)
that  
$$
(\imath_{Y} B^+_{i,j})(w)=B^+_{i,j}(\imath_{Y} w) =0\,, \quad \text{  for all }i\in \{2, \dots, k\}, j\in \{1, \dots, J_i\}\,.
$$ 
Since $w$ is an arbitrary $2$-form of class $C^1$, the Addendum is proved in the case of the
unstable foliation. The statement for the stable foliation follows by considering the case of the unstable foliation 
for the inverse map $A^{-1}$. 

\end{proof} 

Finally, we detail the construction of asymptotic functionals:

\begin{proof}[Proof of Theorem~\ref{thm:Functionals}]
Let $\bar \gamma$ be any unstable arc. We claim that the limit
$$
\beta^+ (\bar \gamma) := \lim_{k\to \infty}  A_\ast^{-k} \mathcal B^+ (  (A^\#)^{[\frac{\log (\mathcal L^+(A^k \bar\gamma))} {h_{top}(A)}]} ( \mathcal C^+(A^k \bar\gamma) )   )$$
exists. In fact, this claim is easily reduced to existence of the limit 
$$
{\mathcal C}_\infty (\bar \gamma):= \lim_{n\to \infty}  \mathcal C^+(A^k \bar\gamma).
$$
It follows from the definition of the map ${\mathcal C}^+ : \Gamma^+ \to E^+$ in formula~\eqref{eq:Cplus}
that there exists a sequence $\{ x_{\ell, m} \vert \ell \in \N, 1\leq m \leq m_\ell\} \subset \bar \gamma$ such that,
for all $k\in \N$, 
$$
\begin{aligned}
{\mathcal C}^+ (A^k \bar \gamma) &=  \sum_{\ell=1}^{n+k} \sum_{m=1}^{m_\ell}  (A^\#)^{-( [ \frac{\log (\mathcal L^+(A^k \bar \gamma))} {h_{top}(A)}] -\ell -k)}  [ \gamma ( A^{-\ell} (x_{\ell,m}) )   ]^+  \\
&= \sum_{\ell=1}^{n+k} \sum_{m=1}^{m_\ell}  (A^\#)^{-( [ \frac{\log (\mathcal L^+(\bar \gamma))} {h_{top}(A)}] -\ell)}  [ \gamma ( A^{-\ell} (x_{\ell,m}) )   ]^+ \,.
\end{aligned} 
$$
Since the cohomology classes $[\gamma ( A^{-\ell} (x_{\ell,m}) )]$ are uniformly bounded, the limit of the above sequence exists and it is equal to the sum of a convergent series, that is,
$$
{\mathcal C}_\infty (\bar \gamma):= \sum_{\ell=1}^{+\infty} \sum_{m=1}^{m_\ell}  (A^\#)^{-( [ \frac{\log (\mathcal L^+(\bar \gamma))} {h_{top}(A)}] -\ell)}  [ \gamma ( A^{-\ell} (x_{\ell,m}) )   ]^+ \,.
$$
The scaling property follows since the definition is well-posed, in fact
$$
\begin{aligned}
\beta^+ (A\gamma) &:= \lim_{k\to \infty}  A_\ast^{-k} \mathcal B^+ (  (A^\#)^{[\frac{\log (\mathcal L^+(A^{k+1} \bar\gamma))} {h_{top}(A)}]} ( \mathcal C^+(A^{k+1} \bar\gamma) )   ) \\ &= A_\ast 
 \lim_{k\to \infty}  A_\ast^{-(k+1)} \mathcal B^+ (  (A^\#)^{[\frac{\log (\mathcal L^+(A^{k+1} \bar\gamma))} {h_{top}(A)}]} ( \mathcal C^+(A^{k+1} \bar\gamma) )   ) =  A_\ast  \beta^+ (\gamma)\,.
\end{aligned}
$$
It is clear by the definition that the map ${\mathcal C}_\infty$ is invariant under (regular) stable holonomies
of unstable arcs. As a consequence, it is possible to extend the functional ${\mathcal B}^+$ to any rectifiable arc by decomposition into finally many paths, which are equivalent to unstable paths under stable holonomies. 

Finally, it follows immediately from Theorem~\ref{thm:asymptotic} that  there exists a constant $C>0$ such that,for any unstable arc $\bar \gamma$, we have
$$
\Vert  \bar \gamma -  \beta^+ (\bar\gamma)   ) \Vert_{-1} \leq C  [\log (1+\mathcal L^+ (\bar \gamma))]^{\max(J^0_A,1)+1}\,.
$$
By the above-mentioned decomposition, we then derive that, for any $\gamma \in \Gamma_r$, 
$$
\Vert  \gamma -  \beta^+ (\gamma)   ) \Vert_{-1} \leq C (1+  \mathcal L^- (\gamma) )  [\log (1+\mathcal L^+ (\gamma))]^{\max(J^0_A,1)+1}\,,
$$
as stated.
\end{proof}

\section{Deviations of ergodic integrals and Ruelle-Pollicott asymptotics}
\label{sec:RAsymptotics}

In this final section, we complete the proofs of Corollary~\ref{cor:Deviations} on deviations of ergodic integrals for unstable vector fields and Corollary~\ref{cor:RAsymptotics} on the Ruelle-Pollicott asymptotics.

\smallskip
We first verify our claim that, as a consequence of Addendum~\ref{addendum:basic}, the measure 
$D_1^X =B_1^+\wedge \hat X$ and, for all $i\in \{2, \dots, k\}$ and all $j\in\{1, \dots, J_i\}$,  the distributions $D_{i,j}^X =B_{i,j} ^+\wedge \hat X$ are $X$-invariant. Let $B$ any basic current of degree and dimension~$1$ for the unstable foliation on $M\setminus \Sigma$. The current $B \wedge \hat X$ is a current of degree~$2$ and dimension~$0$.  Since the contraction operator $\imath_X$ is surjective onto the space of functions, it is
enough to prove that $\imath_X \mathcal L_X (B \wedge \hat X)=0$, as the latter identity then implies 
$\mathcal L_X (B \wedge \hat X)=0$. Indeed, since $\imath_X B=dB=0$ and $\imath_X \hat X=1$ we have
$$
\imath_X \mathcal L_X (B \wedge \hat X) = \imath_X  d  \imath_X (B \wedge \hat X) = 
- \imath_X  d  B =0\,.
$$
Thus our claim is proved. We then proof our result on deviations of ergodic averages.

\begin{proof}[Proof of Corollary~\ref{cor:Deviations}]
For all $x\in M\setminus \Sigma$, let $\gamma^X_T(x)$ denote the orbit of the unstable vector flow $h^X_\R$ on $M\setminus \Sigma$ (which we assume defined).  Let $\mathcal L_T(x)$ denote the Margulis unstable measure
of the orbit $\gamma^X_T(x)$, that is, according to the notations in the statement of Theorem~\ref{thm:Deviations} we have 
$$
\gamma^X_T(x) = \gamma_{{\mathcal L}_T(x)} (x) \,.
$$
For any $f \in C^1(M)$, let $\eta_f :=  f \hat X$. By Theorem~\ref{thm:Deviations}, formula \eqref{eq:deviations} and our definitions, since
$$
\int_{\gamma_{\mathcal L}(x)} \eta_f = \int_0^T   f \circ h^X_t (x) dt \,,
$$
we have the expansion (for some constant $C_X>0$):
\begin{multline}
\label{eq:deviations_1}
\vert  \int_0^T   f \circ h^X_t (x) dt  - \mathcal L_T(x)  D^X_1(f)  \\ -\sum_{i=2}^k \sum_{j=1}^{J_i} 
c_{i,j} (x, \mathcal L_T(x)) D_{i, j} ^X(f)  (\log \mathcal L_T(x))^{j-1} \mathcal L_T(x)^{ \frac{ \log \vert \mu_i\vert}{h_{top}(A)}}  \vert  \\  \leq C_X \Vert f \Vert_{C^1(M)}  [\log (1+ \mathcal L_T(x)) ]^{\max(J^0_A,1) +1}\,.
\end{multline}
For the case of the constant function $f=1$ we have that, for some constant $C'_X>0$,
\begin{multline}
\label{eq:deviations_2}
\vert T  - \mathcal L_T(x)  D^X_1(1) -  \sum_{i=2}^k \sum_{j=1}^{J_i} 
c_{i,j} (x, \mathcal L_T(x)) D_{i, j} ^X(1)  (\log \mathcal L_T(x))^{j-1} \mathcal L_T(x)^{ \frac{ \log \vert \mu_i\vert}{h_{top}(A)}} \vert  \\  \leq C'_X  [\log (1+ \mathcal L_T(x)) ]^{\max(J^0_A,1) +1}\,.
\end{multline}
The statement then follows from formulas~\eqref{eq:deviations_1} and~\eqref{eq:deviations_2}. Indeed, by
formula~\eqref{eq:deviations_2} the ratio  $\mathcal L_T(x)/T$ is uniformly bounded from above and below,
so that the coefficients
$$
c^X_{ij} (x, T):= c_{i,j} (x, \mathcal L_T(x))   ( \frac{\log (\frac{\mathcal L_T(x)}{T})}{\log T}  + 1)^{j-1}( \frac{\mathcal L_T(x)}{T})^{ \frac{ \log \vert \mu_i\vert}{h_{top}(A)}} 
$$
are uniformly bounded from above,  and from formula~\eqref{eq:deviations_1}, for any function $f$ of class $C^1$, of zero average with respect to the invariant measure $\mu_X$, we have  
\begin{multline}
\vert  \int_0^T   f \circ h^X_t (x) dt   -\sum_{i=2}^k \sum_{j=1}^{J_i} 
c^X_{i,j} (x,T)D_{i, j} ^X(f) (\log T)^{j-1} T^{ \frac{ \log \vert \mu_i\vert}{h_{top}(A)}}  \vert  \\ \leq C_X \Vert f \Vert_{C^1(M)}  [\log (1+ \frac{\mathcal L_T(x)}{T} T) ]^{\max(J^0_A,1) +1}\,.
\end{multline}
For general functions of class $C^1$ the conclusion follows by considering their projections onto the subspace
of functions of zero average.

\end{proof}

\begin{proof}[Proof of Corollary~\ref{cor:RAsymptotics}]
Since $A$ is of class $C^r$ with $r>2$, its unstable foliation is of class $C^{1+\alpha}$ for $\alpha>0$.
Let $\phi^+: M\times \R \to M$ the flow of class $C^{1+\alpha}$ defined by following the (one-dimensional) unstable leaves with the speed determined by the Margulis measure. In more precise terms, given an orientation of the unstable leaves, for every $(x,t) \in M\times \R$ the point $\phi^+_t(x)$ is the unique point which belongs
to the unstable leaf $\mathcal F^+(x)$ of $x$, such that $\mathcal L^+(x, \phi^+_t(x)) =\vert t \vert$  with
$[x, \phi ^+_t(x)]$ oriented in positive or negative direction depending on the sign of $t \not =0$. 

By definition the flow $\phi^+$ preserves the Margulis measure, so that we have
$$
\langle f \circ A^n , g \rangle =  \frac{1}{\sigma} \int_0^\sigma \langle f \circ A^n \circ \phi^+_t , g \circ \phi^+_t\rangle
dt \,.
$$
By integration by parts we have the formula
\begin{equation}
\label{eq:correlations}
\begin{aligned}
\int_0^\sigma \langle f \circ A^n \circ \phi^+_t , g \circ \phi^+_t\rangle
dt &= \langle   \int_0^\sigma f \circ A^n \circ \phi^+_t  dt, g \circ \phi^+_\sigma \rangle \\ & -
 \int_0^t  \langle \int_0^s f \circ A^n \circ \phi^+_\tau d\tau ,  \frac{d}{ds} (g \circ \phi^+_s)\rangle ds
\end{aligned}
\end{equation}
Since the unstable curves are parametrized by the Margulis measure, by change of variables we have
\begin{equation}
\label{eq:change_variables}
\int_0^s f \circ A^n \circ \phi^+_\tau d\tau =  \frac{1}{\lambda^n} \int_0^{\lambda^n s}  
f  \circ \phi^+_\tau \circ A^n d\tau \,.
\end{equation}
Let then $\eta^+$ denote a $1$-form of class $C^{1+\alpha}$ such that $\eta^+ (\frac{d}{dt} \phi^+_t ) \equiv 1$.
The $1$-form $\eta^+$ restricts to the oriented Margulis measure on unstable leaves and can be taken to have kernel tangent to the stable foliation. By this definition we have, for all $x\in M$, $n\in \N$ and $s\in \R$, 
and for all $f\in C^1(M)$, 
$$
\int_0^{\lambda^n s}  f  \circ \phi^+_\tau \circ A^n(x) d\tau = \int_{ \gamma_{\lambda^n s}(A^n(x))}  f \eta^+ \,.
$$
By Theorem~\ref{thm:Deviations} and by formula~\eqref{eq:change_variables} we have
$$
\begin{aligned}
\vert \int_0^{\lambda^n s} f \circ A^n \circ \phi^+_\tau(x) d\tau  &- \frac{1}{\lambda^n}
\sum_{i=1}^k \sum_{j=1}^{J_i} c_{i,j}(x, \lambda^n s)  \log (\lambda^n s)^{j-1}
(\lambda^n s) ^{ \frac{ \log \vert \mu_i\vert}{h_{top}(A)}}  B_{i,j}^+(f \eta^+) \vert  \\ & \leq C \Vert f \Vert_{C^1(M)} 
\frac{ [\log (1+\lambda^n s) ]^{\max(J^0_A,1)+1}} {\lambda^n} \,.
\end{aligned} 
$$
We can assume that the function $f$ has zero average, hence we have 
$$B_1^+(f \eta^+)= \int_M  f  \mathcal M^- \otimes \eta^+=   \int_M  f d  \mathcal M =0\,,$$
 so that 
by inserting the above asymptotics in formula~\eqref{eq:correlations} we derive the desired  expansion
for correlations.

\end{proof}


\begin{thebibliography}{20}


\bibitem[Ba]{Ba} V.~Baladi, There are no deviations for the ergodic averages of the Giullietti-Liverani
horocycle flows on the two-torus. Preprint, arXiv:1907.03453.
\bibitem[Bu14]{Bu14} A.~Bufetov, Limit theorems for translation flows, {\it Ann. of Math.} \textbf{179} (2) (2014), 431--499.
\bibitem[BuFo14]{BuFo14} A.~Bufetov \& G.~Forni, Limit theorems for horocycle flows, {\it Ann. Sci. ENS} 
\textbf{47} (5) (2014), 851--903.
\bibitem[DHL14]{DHL14}  V.~Delecroix, P.~Hubert \& S.~Leli\`evre, Diffusion for the periodic wind-tree model,
{\it Ann. Sci. ENS}  \textbf{47} (6) (2014), 1085--1110.
\bibitem[FGL]{FGL} F.~Faure, S.~Gou\"ezel \& E.~Lanneau,  Ruelle spectrum of linear pseudo-Anosov maps,
{\it Journal de l'\'Ecole polytechnique--Math\'ematiques} \textbf{6} (2019), 811--877.
\bibitem[F02]{F02}  G.~Forni, Deviation of Ergodic Averages for Area-Preserving Flows on Surfaces of Higher Genus, {\it Annals of Mathematics} \textbf{155} (1) (2002), 1--103.
\bibitem[F07]{F07}  G.~Forni, Sobolev regularity of solutions of the cohomological equation, {\it Ergodic Theory and Dynamical Systems}, 1-105. doi:10.1017/etds.2019.108.
\bibitem[FoKa]{FoKa} G. Forni \& A. Kanigowski, Time-Changes of Heisenberg nilflows, in ``Some aspects of the theory of dynamical systems: a tribute to Jean-Christophe Yoccoz (volume II)''  (S.~Crovisier, R.~Krikorian, Carlos Matheus, S.~Senti  editors), {\it Ast\'erisque} \textbf{416} (2020),  253--299.
\bibitem[GL]{GL} P.~Giulietti \& C.~Liverani,  Parabolic dynamics and Anisotropic Banach spaces, 
 {\it Journal Europ. Math. Soc.} \textbf{21} (9) (2019), 2793--2858. Published online: 2019-05-20
 DOI: 10.4171/JEMS/892.
 \bibitem[MY16]{MY16} S. Marmi \& J.-C. Yoccoz,    H\"older Regularity of the Solutions of the Cohomological Equation for Roth Type Interval Exchange Maps, {\it Commun. Math. Phys.}  \textbf{344} (1)  (2016) 117--139. https://doi.org/10.1007/s00220-016-2624-9.
 \bibitem[Yoc10]{Yoc10} J.-C. Yoccoz, Interval exchange maps and translation surfaces, in {\it Homogeneous Flows, Moduli Spaces and Arithmetic}, Clay Math. Proc. \textbf{10}, Amer. Math. Soc., Providence, RI, 2010, 1--69.

\end{thebibliography}
\end{document}